\documentclass[a4paper]{amsart}

\RequirePackage{amsmath} \RequirePackage{amssymb}
\usepackage{amscd,latexsym,amsthm,amsfonts,amssymb,amsmath,amsxtra}
\usepackage[colorlinks=true,urlcolor=blue,citecolor=blue]{hyperref}
\usepackage{color}
\usepackage[all]{xy}
\usepackage[OT2,T1]{fontenc}
\usepackage{bm}
\usepackage{mathtools}
\DeclareSymbolFont{cyrletters}{OT2}{wncyr}{m}{n}
\DeclareMathSymbol{\Sha}{\mathalpha}{cyrletters}{"58}
\DeclarePairedDelimiter{\ceil}{\lceil}{\rceil}
\let\Re\undefined

\DeclareMathOperator{\Re}{Re}

\DeclareMathOperator{\ord}{ord}

\DeclareMathOperator{\GL}{GL}

\begin{document}

    \theoremstyle{plain}
\newtheorem{thm}{Theorem} \newtheorem{cor}[thm]{Corollary}
\newtheorem{thmy}{Theorem}
\renewcommand{\thethmy}{\Alph{thmy}}
\newenvironment{thmx}{\stepcounter{thm}\begin{thmy}}{\end{thmy}}
    \newtheorem{lemma}[thm]{Lemma}  \newtheorem{prop}[thm]{Proposition}
    \newtheorem{conj}[thm]{Conjecture}  \newtheorem{fact}[thm]{Fact}
    \newtheorem{claim}[thm]{Claim}
    \theoremstyle{definition}
    \newtheorem{defn}[thm]{Definition}
    \newtheorem{example}[thm]{Example}
    \newtheorem{exercise}[thm]{Exercise}
    \theoremstyle{remark}
    \newtheorem*{remark}{Remark}

    \newcommand{\BA}{{\mathbb {A}}} \newcommand{\BB}{{\mathbb {B}}}
    \newcommand{\BC}{{\mathbb {C}}} \newcommand{\BD}{{\mathbb {D}}}
    \newcommand{\BE}{{\mathbb {E}}} \newcommand{\BF}{{\mathbb {F}}}
    \newcommand{\BG}{{\mathbb {G}}} \newcommand{\BH}{{\mathbb {H}}}
    \newcommand{\BI}{{\mathbb {I}}} \newcommand{\BJ}{{\mathbb {J}}}
    \newcommand{\BK}{{\mathbb {K}}} \newcommand{\BL}{{\mathbb {L}}}
    \newcommand{\BM}{{\mathbb {M}}} \newcommand{\BN}{{\mathbb {N}}}
    \newcommand{\BO}{{\mathbb {O}}} \newcommand{\BP}{{\mathbb {P}}}
    \newcommand{\BQ}{{\mathbb {Q}}} \newcommand{\BR}{{\mathbb {R}}}
    \newcommand{\BS}{{\mathbb {S}}} \newcommand{\BT}{{\mathbb {T}}}
    \newcommand{\BU}{{\mathbb {U}}} \newcommand{\BV}{{\mathbb {V}}}
    \newcommand{\BW}{{\mathbb {W}}} \newcommand{\BX}{{\mathbb {X}}}
    \newcommand{\BY}{{\mathbb {Y}}} \newcommand{\BZ}{{\mathbb {Z}}}

    \newcommand{\CA}{{\mathcal {A}}} \newcommand{\CB}{{\mathcal {B}}}
    \newcommand{\CC}{{\mathcal {C}}} \renewcommand{\CD}{{\mathcal {D}}}
    \newcommand{\CE}{{\mathcal {E}}} \newcommand{\CF}{{\mathcal {F}}}
    \newcommand{\CG}{{\mathcal {G}}} \newcommand{\CH}{{\mathcal {H}}}
    \newcommand{\CI}{{\mathcal {I}}} \newcommand{\CJ}{{\mathcal {J}}}
    \newcommand{\CK}{{\mathcal {K}}} \newcommand{\CL}{{\mathcal {L}}}
    \newcommand{\CM}{{\mathcal {M}}} \newcommand{\CN}{{\mathcal {N}}}
    \newcommand{\CO}{{\mathcal {O}}} \newcommand{\CP}{{\mathcal {P}}}
    \newcommand{\CQ}{{\mathcal {Q}}} \newcommand{\CR}{{\mathcal {R}}}
    \newcommand{\CS}{{\mathcal {S}}} \newcommand{\CT}{{\mathcal {T}}}
    \newcommand{\CU}{{\mathcal {U}}} \newcommand{\CV}{{\mathcal {V}}}
    \newcommand{\CW}{{\mathcal {W}}} \newcommand{\CX}{{\mathcal {X}}}
    \newcommand{\CY}{{\mathcal {Y}}} \newcommand{\CZ}{{\mathcal {Z}}}

    \newcommand{\RA}{{\mathrm {A}}} \newcommand{\RB}{{\mathrm {B}}}
    \newcommand{\RC}{{\mathrm {C}}} \newcommand{\RD}{{\mathrm {D}}}
    \newcommand{\RE}{{\mathrm {E}}} \newcommand{\RF}{{\mathrm {F}}}
    \newcommand{\RG}{{\mathrm {G}}} \newcommand{\RH}{{\mathrm {H}}}
    \newcommand{\RI}{{\mathrm {I}}} \newcommand{\RJ}{{\mathrm {J}}}
    \newcommand{\RK}{{\mathrm {K}}} \newcommand{\RL}{{\mathrm {L}}}
    \newcommand{\RM}{{\mathrm {M}}} \newcommand{\RN}{{\mathrm {N}}}
    \newcommand{\RO}{{\mathrm {O}}} \newcommand{\RP}{{\mathrm {P}}}
    \newcommand{\RQ}{{\mathrm {Q}}} \newcommand{\RR}{{\mathrm {R}}}
    \newcommand{\RS}{{\mathrm {S}}} \newcommand{\RT}{{\mathrm {T}}}
    \newcommand{\RU}{{\mathrm {U}}} \newcommand{\RV}{{\mathrm {V}}}
    \newcommand{\RW}{{\mathrm {W}}} \newcommand{\RX}{{\mathrm {X}}}
    \newcommand{\RY}{{\mathrm {Y}}} \newcommand{\RZ}{{\mathrm {Z}}}

    \newcommand{\fa}{{\mathfrak{a}}} \newcommand{\fb}{{\mathfrak{b}}}
    \newcommand{\fc}{{\mathfrak{c}}} \newcommand{\fd}{{\mathfrak{d}}}
    \newcommand{\fe}{{\mathfrak{e}}} \newcommand{\ff}{{\mathfrak{f}}}
    \newcommand{\fg}{{\mathfrak{g}}} \newcommand{\fh}{{\mathfrak{h}}}
    \newcommand{\fii}{{\mathfrak{i}}} \newcommand{\fj}{{\mathfrak{j}}}
    \newcommand{\fk}{{\mathfrak{k}}} \newcommand{\fl}{{\mathfrak{l}}}
    \newcommand{\fm}{{\mathfrak{m}}} \newcommand{\fn}{{\mathfrak{n}}}
    \newcommand{\fo}{{\mathfrak{o}}} \newcommand{\fp}{{\mathfrak{p}}}
    \newcommand{\fq}{{\mathfrak{q}}} \newcommand{\fr}{{\mathfrak{r}}}
    \newcommand{\fs}{{\mathfrak{s}}} \newcommand{\ft}{{\mathfrak{t}}}
    \newcommand{\fu}{{\mathfrak{u}}} \newcommand{\fv}{{\mathfrak{v}}}
    \newcommand{\fw}{{\mathfrak{w}}} \newcommand{\fx}{{\mathfrak{x}}}
    \newcommand{\fy}{{\mathfrak{y}}} \newcommand{\fz}{{\mathfrak{z}}}
     \newcommand{\fA}{{\mathfrak{A}}} \newcommand{\fB}{{\mathfrak{B}}}
    \newcommand{\fC}{{\mathfrak{C}}} \newcommand{\fD}{{\mathfrak{D}}}
    \newcommand{\fE}{{\mathfrak{E}}} \newcommand{\fF}{{\mathfrak{F}}}
    \newcommand{\fG}{{\mathfrak{G}}} \newcommand{\fH}{{\mathfrak{H}}}
    \newcommand{\fI}{{\mathfrak{I}}} \newcommand{\fJ}{{\mathfrak{J}}}
    \newcommand{\fK}{{\mathfrak{K}}} \newcommand{\fL}{{\mathfrak{L}}}
    \newcommand{\fM}{{\mathfrak{M}}} \newcommand{\fN}{{\mathfrak{N}}}
    \newcommand{\fO}{{\mathfrak{O}}} \newcommand{\fP}{{\mathfrak{P}}}
    \newcommand{\fQ}{{\mathfrak{Q}}} \newcommand{\fR}{{\mathfrak{R}}}
    \newcommand{\fS}{{\mathfrak{S}}} \newcommand{\fT}{{\mathfrak{T}}}
    \newcommand{\fU}{{\mathfrak{U}}} \newcommand{\fV}{{\mathfrak{V}}}
    \newcommand{\fW}{{\mathfrak{W}}} \newcommand{\fX}{{\mathfrak{X}}}
    \newcommand{\fY}{{\mathfrak{Y}}} \newcommand{\fZ}{{\mathfrak{Z}}}
    \newcommand{\Supp}{\operatorname{Supp}}
    \newcommand{\coker}{\operatorname{coker}}
    \newcommand{\Ad}{\operatorname{Ad}}
    \newcommand{\Sw}{\operatorname{Sw}}
    \newcommand{\Sp}{\operatorname{Sp}}
    \newcommand{\Nrd}{\operatorname{Nrd}}
    \newcommand{\Id}{\operatorname{Id}}
    \newcommand{\A}{\operatorname{A}}
    \newcommand{\Irr}{\operatorname{Irr}}
    \newcommand{\JL}{\operatorname{JL}}
    \newcommand{\codim}{\operatorname{codim}}
     \newcommand{\Ind}{\operatorname{Ind}}

\title[]{A constraint for twist equivalence of cusp forms on GL$(n)$}%
\author{Dinakar Ramakrishnan and Liyang Yang}

\address{253-37 Caltech, Pasadena\\
	CA 91125, USA}
\email{dinakar@caltech.edu}
\address{253-37 Caltech, Pasadena\\
CA 91125, USA}
\email{lyyang@caltech.edu}

\date{\today}%
\maketitle

This Note owes its existence to a question posed to the first author by Kaisa Matom\"aki, spurred in turn by her ongoing joint works with Maksym Radziwill on the sign changes of coefficients $\lambda_n(h)$ of cusp forms $h$ (ref. \cite{MR19}). Her question is this:

{\it Suppose $f, g$ are newforms, holomorphic or of Maass type, on the upper half plane $\mathcal H$ of levels $N, N'$ respectively. Is there an optimal upper bound on the conductor $Q$ of a Dirichlet character $\chi$ such that $f, g$ are twist equivalent via $\chi$, i.e., for all but a finite number of primes $p$, we have $\lambda_p(f)=\lambda_p(g)\chi(p)$?}

The answer is Yes, and we could show that $Q\leq \sqrt{NN'}.$ (The fact that $Q$ is bounded above by some power of $NN'$ is not difficult; see the beginning of Section \ref{sec3}.) Now suppose $f$ or $g$ has trivial character. Then we even show that $Q \leq \sqrt{[N,N']}$, where $[N,N']$ is the least common multiple of $N, N'$.
By the Atkin-Lehner theory, one knows that prime divisors of $Q$ must also divide $[N,N']$. In fact our proof below will show that $Q^2$ divides $[N,N']$, still with $f$ or $g$ of trivial character. Moreover, this is optimal, even when $N,$ $N'$ have common divisors. In the special sub-case where $N, N'$ are square-free, $Q$ must be $1$ and there is no non-trivial twist equivalence between $f$ and $g$.

It is convenient for us to work in the framework of automorphic representations, and we prove an analogous result for cusp forms on GL$(n)$, for any $n \geq 1$, over any global field $F$. Denote by $\mathbb{A}_F$ the ring of adeles of $F$.

Let $\pi_1$ and $\pi_2$ be unitary cuspidal representations on $GL(n,\mathbb{A}_F),$ where $F$ is a global field. Let $N_1$ (resp. $N_2$) be the arithmetic conductor of $\pi_1$ (resp. $\pi_2$), which is an ideal in $\mathcal O_F$. Let $\chi$ be an idele class character of $F$. Denote by $Q$ the conductor of $\chi.$

\medskip

\begin{thmx}\label{A}
Let notation be as above. Suppose $\pi_1\otimes\chi\simeq\pi_2.$ Then we have
\begin{equation}\label{0}
Q^n\mid N_1N_2.
\end{equation}
Suppose further that for every finite place $v,$ when $\pi_{1,v}$ or $\pi_{2,v}$ is ramified either it is in the discrete series or $n=2$ with trivial central character. Then $Q^n\mid [N_1, N_2],$ where $[N_1, N_2]$ is the least common multiple of $N_1$ and $N_2.$
\end{thmx}

\medskip
In fact, the proof will give a corresponding local statement at each place, and the case of ramified principal series implies the optimality of the bound.

\medskip

\begin{remark}
Suppose $n=2$ and $\pi_1$ has trivial central character. Then $Q^2\mid [N_1, N_2].$ (See Proposition \ref{32} of Sec. \ref{sec3}.) If moreover, $N_1,$ $N_2$ are squarefree, then $Q=1,$ and $\chi$ a class group character. Thus if $F$ has class number 1, e.g., $F=\mathbb{Q},$ there is no such $\chi.$
\medskip

Also, for $GL(2),$ it was shown by the first author ( \cite{Ram00}), using multiplicity one for $SL(2),$ that if $\Ad(\pi)\simeq\Ad(\pi'),$ where $\Ad$ is the Gelbart-Jacquet adjoint lifting from GL$(2)$ to GL$(3)$, then $\pi'\simeq\pi\otimes\chi$ for some character $\chi.$ So Theorem \ref{A} elucidates the fibre of $\pi\rightarrow\Ad(\pi).$
\end{remark}

\medskip

Our proof of Theorem A reduces it, by the factorizability of the conductors, to a local statement, and further by the Bernstein-Zelevinsky classification, to a statement about discrete series representations of GL$(n, F_v)$, and then makes use of the local Jacquet-Langlands correspondence to a question about twists of (finite-dimensional) irreducible representations of its inner form, the (multiplicative group of the) division algebra $D$ of dimension $n^2$ over $F_v$ of invariant $1/n$, where we appeal to some known results of Koch and Zink. (One can also work in the framework of Moy and Prasad.)

\medskip

In our result, $\chi$ need not have finite order. Thus it is natural to ask for such a bound for the full analytic conductor, which includes an archimedean analysis. This is done in the last section, see Theorem \ref{B} therein. 
\medskip

D. Prasad has referred us to the preprint (\cite{Cor17}), which also treats character twists, but goes in another direction. 

\medskip
We thank Kaisa Matom\"aki, Dipendra Prasad and Maksym Radziwill for their interest.

\medskip

\section{Representations of Central Division Algebras}

Let $v$ be a non-archimedean place of $F.$ Let $\mathfrak{o}_{F_v}$ be the ring of integers of $F_v.$ Let $\varpi_{v}$ be a fixed uniformizer. Let $\ord_{F_v}$ be the valuation normalized such that $\ord_{F_v}(\varpi)=1.$ Denote by $U_k(F_v)=1+\varpi^k\mathfrak{o}_{F_v}$ if $k\geq 1.$ Set $U_0(F_v)=\mathfrak{o}_{F_v}^{\times}.$

Let $\chi_v$ be a quasi-character of $F_v^{\times}.$ Denote by $\A(\chi_v)$ the conductor exponent, i.e., the least integer $k\geq 0$ such that $\chi_v\mid_{U_k(F_v)}=\Id.$

Let $D_v$ be a division algebra over $F_v$ of dimension $[D_v: F_v]=d_v^2.$ Denote by $\Nrd$ the reduced norm on $D_v.$ Let $\ord_{F_v}$ be the normalized valuation on $F_v.$ Denote by
\begin{align*}
U_k(D_v)=\big\{x\in D^{\times}:\ \ord_{F_v}(\Nrd(x-1))\geq k\big\},\quad\forall\  k\geq 1;
\end{align*}
and set $U_0(D_v)=\ker(\ord_{F_v}\circ\Nrd).$ Let $\pi_v'$ be an irreducible admissible representation of $D_v^{\times}.$ Define the level $l(\pi_v')$ of $\pi_v'$ as following:
\begin{align*}
l(\pi_v')=\min_{k\leq 0}\big\{k:\ \pi'_v\mid{U_{k}(D_v)}=\Id\big\}.
\end{align*}

Let $\chi_v$ be a quasi-character of $F_v^{\times}.$ Set $\pi_v'\otimes\chi_v=\pi_v'\otimes(\chi_v\circ\Nrd).$ Then it is clear from the definition that $l(\pi_v'\otimes\chi_v)= l(\chi_v\circ\Nrd),$ assuming $l(\chi_v\circ\Nrd)\geq l(\pi_v').$

\medskip

\subsection{Local Level-Conductor Formula}
Let $\pi'_v$ be an irreducible admissible representation of $D_v^{\times}.$ One can define (e.g.  \cite{KZ80}) the conductor exponent $\A(\pi_v')$ of $\pi_v'$ via the $\epsilon$-factor. Then one has the following level-conductor formula ( (4.3.4) in Sec. 4.3 of \cite{KZ80}), which is well known for $n=2.$
\begin{lemma}
Let $\pi'_v$ be an irreducible admissible representation of $D_v^{\times}.$ Then
\begin{equation}\label{10}
\A(\pi'_v)=l(\pi_v')+d_v-1.
\end{equation}
\end{lemma}
\begin{remark}
One can state \eqref{10} equivalently in terms of the Moy-Prasad depth ( \cite{MP94} and \cite{MP96}).
\end{remark}

\begin{lemma}\label{11}
Let $\pi'_v$ be an irreducible admissible representation of $D_v^{\times}.$ Let $\chi_v$ be a quasi-character of $F_v^{\times}$ such that $l(\chi_v\circ\Nrd)\geq l(\pi_v').$ Then
$\A(\pi_v'\otimes\chi_v)=d_v\cdot\A(\chi_v).$
\end{lemma}
\begin{proof}
Let $x\in F_v^{\times}.$ Then $\ord_{F_v}(\Nrd(x))=\ord_{F_v}(x^{d_v})=d_v\cdot\ord_{F_v}(x).$ Take $k=l(\chi_v\circ\Nrd).$ Then we have
\begin{align*}
\Nrd(U_k(D_v))=U_k(D_v)\cap F_v^{\times}=U_{\ceil{k/d_v}}(F_v),
\end{align*}
where $\ceil{\cdot}$ is the ceiling function. Note that $\chi_v$ is trivial on $U_{\ceil{k/d_v}}(F_v)$ if and only if $\chi_v\circ\Nrd$ is trivial on $U_k(D_v).$ Then by definition, we have
\begin{align*}
\A(\chi_v)\leq \ceil[\bigg]{\frac{l(\chi_v\circ\Nrd)}{d_v}},
\end{align*}
implying that $d_v\cdot (\A(\chi_v)-1)\leq l(\chi_v\circ\Nrd)-1.$ On the other hand, by definition, $l(\chi_v\circ\Nrd)$ is the minimal integer such that $\chi_v$ is trivial on $U_{\ceil{k/d_v}}(F_v).$ Hence
\begin{equation}\label{12}
l(\chi_v\circ\Nrd)=d_v\cdot (\A(\chi_v)-1)+1.
\end{equation}

Since $l(\chi_v\circ\Nrd)\geq l(\pi_v'),$ one has $l(\pi_v'\otimes\chi_v)= l(\chi_v\circ\Nrd).$ Note that $\pi_v'\otimes\chi_v$ is also irreducible. Applying \eqref{10} to $\pi_v'\otimes\chi_v$ we thus deduce
\begin{equation}\label{13}
\A(\pi'\otimes\chi_v)=l(\pi'\otimes\chi_v)+d_v-1.
\end{equation}
Then Lemma \ref{11} follows from \eqref{12} and \eqref{13}.
\end{proof}

\medskip

\subsection{The Local Jacquet-Langlands Correspondence}
We recall briefly the Jacquet-Langlands correspondence here. Let $\Irr(GL(d,F_v))$ (resp. $\Irr(D_v^{\times})$) be the set of equivalence classes of irreducible essentially square-integrable representations of $GL(d,F_v)$ (resp. $D_v^{\times}$). Then there exists a canonical bijection
\begin{align*}
\JL:\ \Irr(GL(d,F_v))\longrightarrow \Irr(D_v^{\times})
\end{align*}
satisfying some properties (e.g.  \cite{ABPS16}, Thm 2.2). Let $\pi_v$ be an irreducible essentially square-integrable representation of $GL(d,F_v).$ Let $\A(\pi_v)$ be the conductor exponent of $\pi_v.$ Denote by $\pi_v'=\JL(\pi_v).$ Then one also has $\A(\pi_v'\otimes\chi_v)=\A(\pi_v\otimes\chi_v),$ for any quasi-character $\chi_v$ of $F_v^{\times}.$

\medskip

\begin{prop}\label{3}
Let notation be as above. Assume $\pi_v$ is irreducible admissible and essentially square-integrable of $\GL(n,F_v)$. Then either
\begin{equation}\label{4}
n\cdot \A(\chi_v)\leq \max\big\{\A(\pi_v\otimes\chi_v), \A(\pi_v)\big\},
\end{equation}
where the equality holds if $\A(\pi_v\otimes\chi_v)\neq \A(\pi_v),$ or $\A(\pi_v)=n\cdot  \A(\chi_v).$
\end{prop}
\begin{proof}
Suppose $n\cdot \A(\chi_v)>\max\{\A(\pi_v\otimes\chi_v), \A(\pi_v)\}.$ Let $\pi_v'=\JL(\pi_v).$ Then $\A(\pi_v)=\A(\pi_v').$ Hence it follows from formulas \eqref{12}, \eqref{13} and the assumption $n\cdot \A(\chi_v)>\A(\pi_v)$ that $l(\chi_v\circ\Nrd)> l(\pi_v').$ We then deduce from Lemma \ref{11} that $\A(\pi_v'\otimes\chi_v)=n\cdot\A(\chi_v).$
Therefore, applying the Jacquet-Langlands correspondence we get  $\A(\pi_v\otimes\chi_v)=n\cdot \A(\chi_v),$ a contradiction! Thus \eqref{4} holds.

Now we suppose $\A(\pi_v\otimes\chi_v)\neq \A(\pi_v).$ Since $\A(\chi_v)=\A(\chi_v^{-1}),$ after twisting by $\chi_v^{-1}$ if needed, we may assume that $\A(\pi_v\otimes\chi_v)>\A(\pi_v),$ which amounts to $\A(\pi_v'\otimes\chi_v)>\A(\pi_v').$ Then by \eqref{13} one has $l(\pi_v'\otimes\chi_v)>l(\pi_v').$ Hence
\begin{equation}\label{14}
l(\chi_v\circ\Nrd)\geq l(\pi_v').
\end{equation}
Suppose further that equality in \eqref{4} does not hold. Then $n\cdot \A(\chi_v)< \A(\pi_v\otimes\chi_v)=\A(\pi_v'\otimes\chi_v).$ Then by Lemma \ref{11} we deduce that $l(\chi_v\circ\Nrd)<l(\pi_v'),$ contradicting \eqref{14}. In all, if $\A(\pi_v\otimes\chi_v)\neq \A(\pi_v),$ then $n\cdot \A(\chi_v)= \max\big\{\A(\pi_v\otimes\chi_v), \A(\pi_v)\big\}.$

Assume $\A(\pi_v)=n\cdot \A(\chi_v).$ Suppose at the same time that $\A(\pi_v\otimes\chi_v)> \A(\pi_v)=n\cdot \A(\chi_v).$ Then we have \eqref{14}, which implies, by Lemma \ref{11}, that $\A(\pi_v\otimes\chi_v)=n\cdot \A(\chi_v),$ contradicting our assumption. Hence $\A(\pi_v\otimes\chi_v)\leq \A(\pi_v).$ Therefore, the equality of \eqref{4} holds when $\A(\pi_v)=n\cdot \A(\chi_v).$
\end{proof}

\begin{remark}
In this Note, we haven't used the local Langlands correspondence. Proposition \ref{3} can be rephrased in terms of Artin conductors of Galois representations. If the Galois representation is irreducible (i.e. $\pi_v$ is supercuspidal), it has been pointed to us that Proposition \ref{3} admits a short proof purely in  
the theory of the Weil-Deligne group by using the following argument: the Artin conductor of V is the integral over $s$  
from $-1$ to infinity of $\codim (V^{ G^s})$, where $G^s$ is the $s$-th higher ramification group in the upper  
numbering filtration of the Galois group. So if V is irreducible, it  
is simply $\dim V$ times the infimum of the set of s where $G_s$ acts  
trivially on V, plus 1. Let $\sigma_v$ be the Weil-Deligne representation associated to $\pi_v$ via local Langlands correspondence. The stated inequality then is equivalent to the statement that if $G_s$  
acts nontrivially on $\chi$, then it acts nontrivially on $\sigma_v$ or  
($\chi \otimes\sigma_v$), which is clear.
\end{remark}

\medskip

\section{Proof of Theorem \ref{A}}\label{sec3}
Let us note, before commencing the proof of Theorem \ref{A}, that it is not hard to check that $Q\leq (N_1N_2)^n.$ By the factorizability of conductors, it suffices to see this at each finite place $v.$ If $\sigma_{i,v},$ $i=1,2,$ is the $n$-dimensional representation of the Weil-Deligne group at $F_v,$ associated to $\pi_{i,v},$ by the local Langlands correspondence then $\sigma_{2,v}\simeq \sigma_{1,v}\otimes\chi_v,$ which implies that $\chi_v\hookrightarrow \check{\sigma_{1,v}}\otimes\sigma_{2,v}$; thus $Q_v\leq (N_{1,v}N_{2,v})^n,$ since the conductor of $\check{\sigma_{1,v}}\otimes\sigma_{2,v}$ is bounded above by $N_{1,v}^nN_{2,v}^n.$ However, we do not appeal to anything from $p$-adic local Langlands correspondence.
\\

For a supercuspidal representation $\sigma$ of $\GL(k,F_v),$ we denote by $\sigma(l)$ the twist of $\sigma$ with the character $|\cdot|^l,$ i.e., the representation $g\mapsto |\det (g)|^l\sigma(g).$ For $m\geq 0,$ set $\Delta(\sigma,m)=[\sigma_j,\sigma_j(1),\cdots,\sigma_j(m_j-1)].$ Let $\sigma'$ be another supercuspidal representation of $\GL(k,F_v),$ and $m'\geq 0.$ We say $\Delta(\sigma,m)$ \textit{precedes} $\Delta(\sigma',m')$ if $\sigma'\cong \sigma(m).$

\medskip 

Let $\pi(\sigma,m)$ be the irreducible representation of $\GL(km,F_v)$ induced from $\sigma\times\sigma(1)\times\cdots\times\sigma(m-1).$ Denote by $Q(\Delta(\sigma,m))$ the unique irreducible quotient of $\pi(\sigma,m).$ Then by the Bernstein-Zelevinsky classification (see \cite{Zel80}), any smooth irreducible representation of $\GL(n,F_v)$ is isomorphic to the unique irreducible quotient $Q(\Delta(\sigma_1,m_1),\cdots,\Delta(\sigma_r,m_r))$ of $Q(\Delta(\sigma_1,m_1))\times\cdots\times Q(\Delta(\sigma_r,m_r)),$ for some supercuspidal representations $\sigma_j,$ such that $\Delta(\sigma_{j_1},m_{j_1})$ does not precede $\Delta(\sigma_{j_2},m_{j_2})$ for $j_1<j_2.$ We now use this classification to prove our main theorem.

\begin{proof}[Proof of Theorem \ref{A}]

\medskip 
	
Write $\pi_i=\otimes_v'\pi_{i,v},$ $1\leq i\leq 2.$ Let $v$ be a nonarchimedean place. Then $\pi_{i,v}$ is a smooth irreducible representation of $\GL(n,F_v).$ Hence, by Bernstein-Zelevinsky classification, $\pi_{1,v}$ is isomorphic to a representation of the form $Q(\Delta_1,\cdots,\Delta_r)$ for a unique (up to permutation) collection of intervals $\Delta_j=\Delta(\sigma_j,m_j),$ with $\sigma_j$ a supercuspidal representation of some $\GL(n_j,F_v),$ $1\leq j\leq r,$ $\sum_jm_jn_j=n;$ such that $\Delta_i$ does not precede $\Delta_j$ for $i<j.$ Moreover, each $Q(\Delta_j)$ is essentially square-integrable.

Then it follows from $\pi_{2,v}\simeq\pi_{1,v}\otimes\chi_v$ and uniqueness of the quotient, that $\pi_{2,v}$ is isomorphic to a representation of the form $Q(\Delta_1',\cdots,\Delta_r'),$ where $\Delta_j'=\Delta(\sigma_j\otimes\chi_v,m_j),$ $1\leq j\leq r.$ Then by Proposition \ref{3} we have, for each $1\leq i\leq r,$ that
\begin{align*}
m_jn_j\cdot \A(\chi_v)\leq \max\big\{\A(Q(\Delta_j)\otimes\chi_v), \A(Q(\Delta_j))\big\},
\end{align*}
namely, $m_jn_i\cdot \A(\chi_v)\leq\max\big\{\A(Q(\Delta_j')), \A(Q(\Delta_j))\big\}.$ Summing through all $1\leq j\leq r$ one then obtain
\begin{align*}
\sum_{j=1}^rm_jn_j\cdot \A(\chi_v)\leq \sum_{j=1}^r \max\big\{\A(Q(\Delta_j')), \A(Q(\Delta_j))\big\}.
\end{align*}
Note that $\sum_{j=1}^rm_jn_j=n.$ Hence,
\begin{equation}\label{a}
n\A(\chi_v)\leq \sum_{j=1}^r \max\big\{\A(Q(\Delta_j')), \A(Q(\Delta_j))\big\},
\end{equation}
from which we deduce that $Q_v^n$ divides $N_{1,v}N_{2,v}.$ Thus \eqref{0} follows.
\medskip

Suppose further for every finite place $v,$ $\pi_{1,v}$ is a discrete series whenever $\pi_{1,v}$ is ramified. Let $v$ be a nonarchimedean place as above. Then $r=1.$ So we have by \eqref{a} that $n\A(\chi_v)\leq\max\big\{\A(Q(\Delta_1')), \A(Q(\Delta_1))\big\}=\max\{\A(\pi_{1,v}), \A(\pi_{2,v})\},$ from which we deduce that $Q_v^n$ divides $[N_{1,v}, N_{2,v}].$ Now Theorem \ref{A} follows from the above analysis and the following claim:
\begin{claim}\label{32}
Let notation be as before. Assume $n=2$ and $\pi_1$ has trivial central character. Then $Q^2\mid [N_1, N_2].$
\end{claim}
\end{proof}
\begin{proof}[Proof of Claim \ref{32}]
Let notation be as in the proceeding proof. From the above proof we know that $2\A(\chi_v)\leq \max\{\A(\pi_{1,v}), \A(\pi_{2,v})\}$ if $r=1.$ Hence, it suffices to consider the case when $r=2,$ in which case $\pi_{i,v}$ is a principal series, $1\leq i\leq 2.$ We may write $\pi_{1,v}=\chi_{1,v}\boxplus\chi_{2,v}$ and thus $\pi_{2,v}\simeq \chi_v\chi_{1,v}\boxplus\chi_v\chi_{2,v}.$
\medskip

On the other hand, since $\pi_1$ has trivial central character, $\A(\chi_{1,v})=\A(\chi_{2,v}).$ Suppose $\A(\chi_v)\leq \A(\chi_{1,v}).$ Then $2\A(\chi_v)\leq \A(\chi_{1,v})+\A(\chi_{2,v})=\A(\pi_{1,v}),$ which is no more than $\max\{\A(\pi_{1,v}), \A(\pi_{2,v})\}.$ Therefore, we may suppose $\A(\chi_v)> \A(\chi_{1,v}).$ By Proposition \ref{3} we deduce that $\A(\chi_v)\leq \max\{\A(\chi_{j,v}),\A(\chi_v\chi_{j,v})\},$ for $1\leq j\leq 2;$ moreover, since $\A(\chi_{1,v})=\A(\chi_{2,v})<\A(\chi_v)\leq \max\{\A(\chi_{j,v}),\A(\chi_v\chi_{j,v})\},$ $j=1, 2,$ then from the condition when the equality holds therein, we obtain that $\A(\chi_v)=\A(\chi_v\chi_{j,v}),$ for $j=1, 2.$ Therefore, we deduce that 
\begin{align*}
2\A(\chi_v)\leq \A(\chi_v\chi_{1,v})+\A(\chi_v\chi_{2,v})=\A(\pi_{2,v})\leq \max\{\A(\pi_{1,v}), \A(\pi_{2,v})\}.
\end{align*}
Hence Claim \ref{32} follows.
\end{proof}

This completes the proof of Theorem \ref{A}. Moreover, towards the extreme cases of Theorem \ref{A}, we have the following:
\begin{itemize}
\item[1.] Let notation be as before. Let $\pi_1=\Ind(\chi_1,\chi_2)$ be an induced representation of $GL(2,\mathbb{A}_{\mathbb{Q}}).$ Let $p_1, p_2$ be two distinct primes. Suppose $\chi_i$ has arithmetic conductor $p_i,$ $1\leq i\leq 2.$ Let $\chi=\chi_1^{-1}\chi_2^{-1},$ and $\pi_2=\pi_1\otimes\chi.$ Then both $\pi_1$ and $\pi_2$ have arithmetic conductor $p_1p_2,$ namely, $N_1=N_2=p_1p_2.$ Also, $\chi$ has arithmetic conductor $Q=p_1p_2.$ And we thus have
$Q^2=N_1N_2.$
\medskip
\item[2.]
Let $\pi_1$ be a cuspidal representation on $GL(n,\mathbb{A}_F)$ such that for every finite place $v,$ $\pi_{1,v}$ is a discrete series whenever $\pi_{1,v}$ is ramified. Suppose further that $n\mid \A(\pi_{1,v}),$ for each finite place $v.$ Let $\chi$ be an idele class character of Artin conductor exponent $\A(\chi_v)=\A(\pi_{1,v})/n,$ for any finite place $v.$ Let $\pi_2=\pi_1\otimes \chi.$ Then it follows from Proposition \ref{3} that $n\A(\chi_v)=\max\{\A(\pi_{1,v}), \A(\pi_{2,v})\},$ for all finite place $v.$ Thus $Q^n=[N_1, N_2]$ in this case.
\end{itemize}

\section{Comparison of Analytic Conductors}
Let $\pi$ be unitary cuspidal representations on $GL(n,\mathbb{A}_F),$ where $F$ is a global field. Denote by $\Sigma$ the set of places of $F.$ Then $\pi=\otimes'_v\pi_{v}.$ Let $L(s,\pi)$ be the principal $L$-function associated to $\pi.$ Then $L(s,\pi)=\prod_{v\in\Sigma}L_v(s,\pi_v),$ where $\Re(s)\gg0.$ Let $\Sigma_{\infty}$ be the set of archimedean places of $F.$ To define analytic conductor of $\pi,$ we need to recall the definition of each $L_v(s,\pi_v)$ for $v\in\Sigma_{\infty}.$

In this section,we fix a place $v\in\Sigma_{\infty},$ denote by $\sigma_v$ the $n$-dimensional Weil-Deligne representation corresponding to $\pi_v$ under the local Langlands correspondence. Then $L_v(s,\pi_v)=L_v(s,\sigma_v).$ Moreover, let $\chi_v$ be a quasi-character of $F_v^{\times},$ we have $L_v(s,\pi_v\times\chi_v)=L_v(s,\sigma_v\otimes\chi_v).$ Hence, it suffices to recall the definition of $L_v(s,\sigma_v\otimes\chi_v),$ which is a product of Gamma functions. Write
\begin{align*}
\sigma_{v}=\oplus_{j=1}^{r_v}\sigma_{v,j},
\end{align*}
where each $\sigma_{v,j}$ is an irreducible representation of the Weil group $\mathcal{W}_{F,v}.$ Hence,
\begin{equation}\label{15}
L_v(s,\sigma_v\otimes\chi_v)=\prod_{j=1}^{r_v}L_v(s,\sigma_{v,j}\otimes\chi_v).
\end{equation}

To define archimedean conductor, we shall describe each $L_v(s,\sigma_{v,j}\otimes\chi_v)$ explicitly. Since our approach is using Langlands classification ( \cite{Del73}), we will separate the cases when $F_{v}\simeq \mathbb{R}$ and $F_{v}\simeq \mathbb{C}.$
\begin{itemize}
	\item[Case 1:] Assume that $F_v\simeq \mathbb{C}.$ One has $\mathcal{W}_{F,v}\simeq\mathbb{C}^{\times}.$ So all irreducible representations are one dimensional. We may write any such characters as $\tau_{k,\nu}(z)=(z/|z|)^k|z|_{\mathbb{C}}^{\nu}=(z/|z|)^k|z|^{2\nu},$ for $k\in \mathbb{Z}$ and $\nu\in\mathbb{C}.$ The local $L$-function associated to this character is $L_v(s,\tau_{k,\nu})=\Gamma_{\mathbb{C}}(s+\nu+|k|/2),$ where $\Gamma_{\mathbb{C}}(s):=2(2\pi)^{-s}\Gamma(s).$
	
	Let $\sigma_{v,j}=\tau_{k_j,\nu_j},$ $k_j\in \mathbb{Z}$ and $\nu_j\in\mathbb{C},$ $1\leq j\leq r.$ Also, one can write $\chi_v$ as $\tau_{k',\nu'}$ for some $k'\in \mathbb{Z}$ and $\nu'\in\mathbb{C}.$ Since $\tau_{k,\nu}\otimes \tau_{k',\nu'}=\tau_{k+k',\nu+\nu'},$ we then have
\begin{align*}
L_v(s,\sigma_{v,j}\otimes\chi_v)=L_v(s,\tau_{k_j+k',\nu_j+\nu'})=\Gamma_{\mathbb{C}}(s+\nu_j+\nu'+|k_j+k'|/2).
\end{align*}
Define the archimedean conductor of $\sigma_{v,j}\otimes\chi_v$ in this case to be
\begin{equation}\label{16}
C_v(\sigma_{v,j}\otimes\chi_v)=(1+|\nu_j+\nu'+|k_j+k'|/2|)^2.
\end{equation}
\medskip
	\item[Case 2:] Assume that $F_v\simeq \mathbb{R}.$ One has $\mathcal{W}_{F,v}\simeq\mathbb{C}^{\times}\sqcup\boldsymbol{j}\mathbb{C}^{\times},$ where $\boldsymbol{j}^2=-1$ and $\boldsymbol{j}z\boldsymbol{j}^{-1}=\bar{z}$ for any $z\in\mathbb{C}^{\times}.$ Hence each irreducible representation $\sigma_{v,j}$ of $\mathcal{W}_{F,v}$ is of dimension 1 or 2.
	\begin{itemize}
	\item[(a)] If $\dim\sigma_{v,j}=1,$ then its restriction to $\mathbb{C}^{\times}$ is of the form $\tau_{0,\nu_j}$ for some $\nu_j\in\mathbb{C}$ ( (3.2) of \cite{K94}). Also, we can write $\chi_v=\tau_{0,\nu'}$ for some $\nu'\in\mathbb{C}.$ In this case, we have
	\begin{align*}
	L_v(s,\sigma_{v,j}\otimes\chi_v)=\Gamma_{\mathbb{R}}(s+\nu_j+\nu'+(1-\sigma_{v,j}(\boldsymbol{j})\chi_v(\boldsymbol{j}))/2),
	\end{align*}
	where $\Gamma_{\mathbb{R}}(s):=\pi^{-s/2}\Gamma(s/2).$
	Define the archimedean conductor of $\sigma_{v,j}\otimes\chi_v$ in this case to be
	\begin{equation}\label{17}
C_v(\sigma_{v,j}\otimes\chi_v)=1+|\nu_j+\nu'+(1-\sigma_{v,j}(\boldsymbol{j})\chi_v(\boldsymbol{j}))/2|.
	\end{equation}
	
	\item[(b)] If $\dim\sigma_{v,j}=2,$ we may assume that $\sigma_{v,j}$ is induced from $\mathbb{C}^{\times}$ to $GL(2,\mathbb{R})$ by $\tau_{k_j,\nu_j},$ where $k_j\in\mathbb{N}_{\geq 1}$ and $\nu_j\in\mathbb{C}.$ Then $\sigma_{v,j}\otimes\chi_v$ is induced from $\mathbb{C}^{\times}$ by $\tau_{k_j,\nu_j+\nu'}.$ The $L$-factor is defined to be
	\begin{align*}
	L_v(s,\sigma_{v,j}\otimes\chi_v)=\Gamma_{\mathbb{C}}(s+\nu_j+\nu'+k_j/2).
	\end{align*}
	Define the archimedean conductor of $\sigma_{v,j}\otimes\chi_v$ in this case to be
	\begin{equation}\label{18}
	C_v(\sigma_{v,j}\otimes\chi_v)=(1+|\nu_j+\nu'+k_j/2|)^2.
	\end{equation}
	\end{itemize}
	\end{itemize}
\medskip
\begin{defn}
Let notation be as above. Define the \textit{archimedean conductor} of $\pi_v\otimes\chi_v$ to be
\begin{align*}
C_v(\pi_v\otimes\chi_v)=C_v(\sigma_v\otimes\chi_v)=\prod_{j=1}^{r_v}C_v(\sigma_{v,j}\otimes\chi_v),
\end{align*}
where each $C_v(\sigma_{v,j}\otimes\chi_v)$ is defined via \eqref{16}, \eqref{17} or \eqref{18}. Let $N(\pi\times\chi)$ be the arithmetic conductor of $\pi\times\chi.$ We set
\begin{align*}
C_{\infty}(\pi\times\chi)=\prod_{v\in\Sigma_{\infty}}C_v(\pi_v\otimes\chi_v)
\end{align*}
to be the archimedean conductor of $\pi\times\chi.$ And let
$$C(\pi\times\chi)=N(\pi\times\chi)\cdot C_{\infty}(\pi\times\chi)$$ be the \textit{analytic conductor} of $\pi\times\chi.$
\end{defn}

Recall that $\pi_1$ and $\pi_2$ are unitary cuspidal representations on $GL(n,\mathbb{A}_F).$ Denote by $\sigma_{i,v}$ the Weil-Deligne representation associated to $\pi_{i,v},$ $1\leq i\leq 2.$ Assume $\pi_1\otimes\chi\simeq\pi_2.$ Then $\sigma_{1,v}\otimes\chi_v\simeq\sigma_{2,v},$ for all $v\in\Sigma_{\infty}.$ Let
\begin{align*}
\sigma_{i,v}=\oplus_{j=1}^{r_v}\sigma_{i,v;j},\quad 1\leq i\leq 2,
\end{align*}
where each $\sigma_{i,v;j}$ is an irreducible representation of $\mathcal{W}_{F,v}.$ Hence,
\begin{equation}\label{19}
\sigma_{1,v;j}\otimes\chi_v=\sigma_{2,v;j},\quad 1\leq j\leq d_v,\  v\in\Sigma_{\infty}.
\end{equation}
\begin{itemize}
	\item[Case 1:] Assume that $F_v\simeq \mathbb{C}.$ One can write $\sigma_{i,v;j}=\tau_{k_{i,j},\nu_{i,j}},$ $k_{i,j}\in \mathbb{Z}$ and $\nu_{i,j}\in\mathbb{C},$ $1\leq j\leq r.$ Also, one can write $\chi_v$ as $\tau_{k',\nu'}$ for some $k'\in \mathbb{Z}$ and $\nu'\in\mathbb{C}.$ Then \eqref{19} yields $k_{1,j}+k'=k_{2,j}$ and $\nu_{1,j}+\nu'=\nu_{2,j}.$ Hence
	\begin{equation}\label{21}
	C_v(\chi_v)=(1+|\nu_{2,j}-\nu_{1,j}+|k_{2,j}-k_{1,j}|/2|)^2.
	\end{equation}
	\medskip
	\item[Case 2:] Assume that $F_v\simeq \mathbb{R}.$ If $\dim\sigma_{i,v;j}=1,$ then its restriction to $\mathbb{C}^{\times}$ is of the form $\tau_{0,\nu_{i,j}}$ for some $\nu_{i,j}\in\mathbb{C}.$ Also, we can write $\chi_v=\tau_{0,\nu'}$ for some $\nu'\in\mathbb{C}.$ In this case, we have $\nu_{1,j}+\nu'=\nu_{2,j}.$ Hence
	\begin{equation}\label{22}
	C_v(\chi_v)=1+|\nu_{2,j}-\nu_{1,j}+(1-\sigma_{2,v;j}\sigma_{1,v;j}^{-1}(\boldsymbol{j}))/2|.
	\end{equation}
	If $\dim\sigma_{i,v;j}=2,$ we may assume that $\sigma_{i,v;j}$ is induced from $\mathbb{C}^{\times}$ to $GL(2,\mathbb{R})$ by $\tau_{k_{i,j},\nu_{i,j}},$ where $k_{i,j}\in\mathbb{N}_{\geq 1}$ and $\nu_{i,j}\in\mathbb{C}.$ Then $\sigma_{1,v;j}\otimes\chi_v$ is induced from $\mathbb{C}^{\times}$ by $\tau_{k_{1,j},\nu_{1,j}+\nu'}.$ Hence $k_{1,j}=k_{2,j}$ and $\nu_{1,j}+\nu'=\nu_{2,j},$ implying
\begin{equation}\label{23}
C_v(\chi_v)=1+|\nu_{2,j}-\nu_{1,j}+(1-\sigma_{2,v;j}\sigma_{1,v;j}^{-1}(\boldsymbol{j}))/2|.
\end{equation}
\end{itemize}
\medskip
\begin{lemma}\label{24}
Assume $F_v\simeq \mathbb{R}.$ Then
\begin{equation}\label{25}
C_v(\chi_v)\leq 3\cdot [C_v(\sigma_{1,v})C_v(\sigma_{2,v})]^{1/n}.
\end{equation}
Moreover, the above bound is sharp.
\end{lemma}
\begin{proof}
Since $\dim\sigma_{1,v;j}=1$ or $2,$ we shall discuss the two cases separately.
\begin{itemize}
\item[Case I.] Suppose $\dim\sigma_{1,v;j}=1.$ Then $\dim\sigma_{2,v;j}=1.$ Hence by \eqref{22} we have
\begin{align*}
C_v(\chi_v)=1+|\nu_{2,j}-\nu_{1,j}+(1-\sigma_{2,v;j}\sigma_{1,v;j}^{-1}(\boldsymbol{j}))/2|.
\end{align*}
\begin{itemize}
	\item[(a).] Assume $\sigma_{2,v;j}(\boldsymbol{j})\sigma_{1,v;j}(\boldsymbol{j})=1.$ If $\sigma_{2,v;j}(\boldsymbol{j})=\sigma_{1,v;j}(\boldsymbol{j})=1,$ then $C_v(\chi_v)=1+|\nu_{2,j}-\nu_{1,j}|\leq 1+|\nu_{2,j}|+|\nu_{1,j}|= 1+|\nu_{2,j}+(1-\sigma_{2,v;j}(\boldsymbol{j}))/2|+|\nu_{1,j}+(1-\sigma_{1,v;j}(\boldsymbol{j}))/2|.$ Thus $C_v(\chi_v)\leq C_v(\sigma_{1,v;j})C_v(\sigma_{2,v;j}).$
	
	\medskip
	Suppose $\sigma_{2,v;j}(\boldsymbol{j})=\sigma_{1,v;j}(\boldsymbol{j})=-1.$ By Corollary 2.5 of \cite{JS81}, $|\Re(\nu_{1,j})|\leq 1/2$ and $|\Re(\nu_{2,j})|\leq 1/2.$ Thus $|\nu_{i,j}|\leq |\nu_{i,j}+1|=|\nu_{i,j}+(1-\sigma_{i,v;j}(\boldsymbol{j}))/2|,$ $1\leq i\leq 2.$ Therefore,  $C_v(\chi_v)=1+|\nu_{2,j}-\nu_{1,j}|\leq 1+|\nu_{2,j}|+|\nu_{1,j}|\leq 1+|\nu_{2,j}+(1-\sigma_{2,v;j}(\boldsymbol{j}))/2|+|\nu_{1,j}+(1-\sigma_{1,v;j}(\boldsymbol{j}))/2|.$ Hence, again, we have  $C_v(\chi_v)\leq C_v(\sigma_{1,v;j})C_v(\sigma_{2,v;j}).$
	\medskip
	
	\item[(b).] Assume $\sigma_{2,v;j}(\boldsymbol{j})\sigma_{1,v;j}(\boldsymbol{j})=-1.$ Then $C_v(\chi_v)=1+|\nu_{2,j}-\nu_{1,j}+1|.$ If $\sigma_{2,v;j}(\boldsymbol{j})=-1,$ then $C_v(\chi_v)\leq 1+|\nu_{2,j}+1|+|\nu_{1,j}|\leq (1+|\nu_{1,j}|)(1+|\nu_{2,j}+1|)=C_v(\sigma_{1,v;j})C_v(\sigma_{2,v;j}).$  If $\sigma_{2,v;j}(\boldsymbol{j})=1,$ then $C_v(\chi_v)=1+|\nu_{2,j}-\nu_{1,j}+1|\leq 1+|\nu_{1,j}+1|+|\nu_{2,j}|+2\leq 3(1+|\nu_{2,j}|)(1+|\nu_{1,j}+1|),$ namely, we have $C_v(\chi_v)=3C_v(\sigma_{1,v;j})C_v(\sigma_{2,v;j}).$
\end{itemize}
 In all, if $\dim\sigma_{1,v;j}=1,$ then we deduce that
 \begin{equation}\label{28}
 C_v(\chi_v)=3\cdot C_v(\sigma_{1,v;j})C_v(\sigma_{2,v;j}).
 \end{equation}
\medskip
\item[Case II.]  Suppose $\dim\sigma_{1,v;j}=2.$ Then $\dim\sigma_{2,v;j}=2.$ Hence by \eqref{23} we have
\begin{align*}
C_v(\chi_v)=1+|\nu_{2,j}-\nu_{1,j}+(1-\sigma_{2,v;j}\sigma_{1,v;j}^{-1}(\boldsymbol{j}))/2|.
\end{align*}
Let $\sigma_{i,v;j}$ be induced from $\mathbb{C}^{\times}$ by $\tau_{k_{i,j},\nu_{i,j}},$ where $k_{i,j}\in\mathbb{N}_{\geq 1}$ and $\nu_{i,j}\in\mathbb{C},$ $1\leq i\leq 2.$ Then $\sigma_{1,v;j}\otimes\chi_v$ is induced from $\mathbb{C}^{\times}$ by $\tau_{k_{1,j},\nu_{1,j}+\nu'}.$ Hence $k_{1,j}=k_{2,j}$ and $\nu_{1,j}+\nu'=\nu_{2,j}.$ Then by triangle inequality we have
\begin{align*}
C_v(\chi_v)&=1+|\nu_{2,j}-\nu_{1,j}+(1-\sigma_{2,v;j}\sigma_{1,v;j}^{-1}(\boldsymbol{j}))/2|\\
&\leq 2+|\nu_{2,j}+k_{2,j}/2|+|\nu_{1,j}+k_{1,j}/2|\\
&\leq 2(1+|\nu_{1,j}+k_{1,j}/2|)(1+|\nu_{2,j}+k_{2,j}/2|).
\end{align*}
Recall that in this case $C_v(\sigma_{i,v;j})=(1+|\nu_{i,j}+k_{i,j}/2|)^2,$ $1\leq i\leq 2.$ So
 \begin{equation}\label{29}
C_v(\chi_v)^2\leq 4\cdot C_v(\sigma_{1,v;j})C_v(\sigma_{2,v;j}).
\end{equation}
\end{itemize}
\medskip
Let $\sigma_{i,v}=\oplus_{j=1}^{r_v}\sigma_{i,v;j}$ be the decomposition of $\sigma_{i,v}$ into irreducible representations. Let $r_{l}$ be the number of $\sigma_{i,v;j}$'s  such that $\dim \sigma_{i,v;j}=l,$ $1\leq l\leq 2.$ Then $r_1+2r_2=n.$ It the follows from \eqref{28} and \eqref{29} that
\begin{align*}
C_v(\chi_v)^n&\leq \prod_{\substack{1\leq j\leq r_v\\ \dim \sigma_{i,v;j}=1}}\big[3\cdot C_v(\sigma_{1,v;j})C_v(\sigma_{2,v;j})\big]\cdot \prod_{\substack{1\leq j\leq r_v\\ \dim \sigma_{i,v;j}=2}}\big[4\cdot C_v(\sigma_{1,v;j})C_v(\sigma_{2,v;j})\big]\\
&\leq 3^{r_1}4^{r_2}\cdot C_v(\sigma_{1,v})C_v(\sigma_{2,v})\leq 3^{n}\cdot C_v(\sigma_{1,v})C_v(\sigma_{2,v}).
\end{align*}
Thus \eqref{25} follows. Moreover, from the above proof, it is clear that the equality in \eqref{25} holds if $r_1=n$ (so $r_2=0$) and $\nu_{1,j}=-1,$ $\nu_{2,j}=0,$ for all $1\leq j\leq r_v.$
\end{proof}

\begin{lemma}\label{26}
Assume $F_v\simeq \mathbb{C}.$ Then
\begin{equation}\label{20}
C_v(\chi_v)\leq 9\cdot [C_v(\sigma_{1,v})C_v(\sigma_{2,v})]^{1/n}.
\end{equation}
Moreover, the above bound is sharp.
\end{lemma}
\begin{proof}
Since $F_v\simeq \mathbb{C},$ we can write $\sigma_{i,v;j}=\tau_{k_{i,j},\nu_{i,j}},$ $k_{i,j}\in \mathbb{Z}$ and $\nu_{i,j}\in\mathbb{C},$ $1\leq j\leq r.$ Write $\chi_v$ as $\tau_{k',\nu'}$ for some $k'\in \mathbb{Z}$ and $\nu'\in\mathbb{C}.$ Then \eqref{19} yields $k_{1,j}+k'=k_{2,j}$ and $\nu_{1,j}+\nu'=\nu_{2,j}.$ So by \eqref{21},  $C_v(\chi_v)=(1+|\nu_{2,j}-\nu_{1,j}+|k_{2,j}-k_{1,j}|/2|)^2.$
By triangle inequality we have
\begin{equation}\label{30}
1+|\nu_{2,j}-\nu_{1,j}+|k_{2,j}-k_{1,j}|/2|\leq 1+|\nu_{2,j}|+|\nu_{1,j}|+|k_{2,j}|/2+|k_{1,j}|/2.
\end{equation}
On the other hand, we have the following inequality:
\begin{claim}\label{31}
Let $1\leq i\leq 2.$ Then $|\nu_{i,j}|+|k_{i,j}|/2\leq 1+3|\nu_{i,j}+|k_{i,j}|/2|.$
\end{claim}

Therefore, by Claim \ref{31}, one obtains the upper bound for $C_v(\chi_v):$
\begin{align*}
C_v(\chi_v)\leq (3+3|\nu_{i,j}+|k_{i,j}|/2|+3|\nu_{i,j}+|k_{i,j}|/2|)^2\leq 9\cdot \prod_{i=1}^2(1+|\nu_{i,j}+|k_{i,j}|/2|)^2.
\end{align*}
Hence, by \eqref{16} we have  $C_v(\chi_v)\leq 9\cdot C_v(\sigma_{1,v;j})C(\sigma_{2,v;j}).$ Invoking with the decomposition $\sigma_{i,v}=\oplus_{j=1}^{r_v}\sigma_{i,v;j}$ we thus get
\begin{align*}
C_v(\chi_v)^n&\leq \prod_{\substack{1\leq j\leq r_v}}\big[9\cdot C_v(\sigma_{1,v;j})C_v(\sigma_{2,v;j})\big]=9^n\cdot C_v(\sigma_{1,v})C_v(\sigma_{2,v}).
\end{align*}
Therefore, \eqref{20} follows.
\end{proof}
\begin{proof}[Proof of Claim \ref{31}]
\begin{itemize}
	\item[(a).] Suppose $k_{i,j}=0.$ Then $|\nu_{i,j}|+|k_{i,j}|/2=|\nu_{i,j}|.$ Hence $|\nu_{i,j}|+|k_{i,j}|/2\leq 1+3|\nu_{i,j}+|k_{i,j}|/2|.$ holds trivially.
\medskip
	
	\item[(b).] Suppose $|k_{i,j}|=1.$ Then by triangle inequality, $|\nu_{i,j}|+|k_{i,j}|/2=|\nu_{i,j}|+1/2\leq |\nu_{i,j}+1/2|+1 \leq 1+3|\nu_{i,j}+|k_{i,j}|/2|.$
	\medskip
	
	\item[(c).] Suppose $|k_{i,j}|\geq 2.$ Note that by Corollary 2.5 of \cite{JS81}, $|\Re(\nu_{i,j})|\leq 1/2.$ Then $|\Re(\nu_{i,j})|\leq |\Re(\nu_{i,j})+|k_{i,j}|/2|,$ as $|k_{i,j}|/2\geq 1.$ Also, $2|\Re(\nu_{i,j})+|k_{i,j}|/2|\geq |k_{i,j}|/2.$ So $|\nu_{i,j}|+|k_{i,j}|/2\leq |\nu_{i,j}+|k_{i,j}|/2|+|k_{i,j}|/2\leq |\nu_{i,j}+|k_{i,j}|/2|+2|\Re(\nu_{i,j})+|k_{i,j}|/2|\leq 1+3|\nu_{i,j}+|k_{i,j}|/2|.$
\end{itemize}
In all, Claim \ref{31} holds.
\end{proof}
\begin{remark}
Our proofs of Lemma \ref{24} and Lemma \ref{26} would imply an explicit version of Lemma A.2 in \cite{HB18} (in the case when $n=n'$). Also, the original proof of Lemma A.2 there is not quite complete as the inequality chain right above (A. 13) (see P. 14 of \cite{HB18}) is not correct for $k=1.$
\end{remark}

\medskip

\begin{thmx}\label{B}
Let $\pi_1$ and $\pi_2$ be unitary cuspidal representations on $GL(n,\mathbb{A}_F)$ such that  $\pi_1\otimes\chi\simeq\pi_2,$ for a character $\chi.$ Let $C_{1}$ (resp. $C_{2}$) be the analytic conductor of $\pi_1$ (resp. $\pi_2$). Let $\chi$ be a Hecke character on $F$. Denote by $C$ the analytic conductor of $\chi.$ Then
\begin{equation}\label{27}
C\leq 3^{[F: \mathbb{Q}]}\cdot (C_1C_2)^{1/n},
\end{equation}
where $[F:\mathbb{Q}]$ is the degree of $F/\mathbb{Q}.$
\end{thmx}
\begin{proof}
Let $r_1$ (resp. $r_2$) be the number of real (complex) places of $F.$ Then $r_1+2r_2=[F: \mathbb{Q}].$ By Lemma \ref{24} and Lemma \ref{26}, one has
\begin{align*}
\prod_{v\in\Sigma_{\infty}}C_v(\chi_v)&\leq \prod_{\substack{v\in\Sigma_{\infty}\\ F_v\simeq \mathbb{R}}}3\cdot [C_v(\sigma_{1,v})C_v(\sigma_{2,v})]^{1/n}\cdot \prod_{\substack{v\in\Sigma_{\infty}\\ F_v\simeq \mathbb{C}}}9\cdot [C_v(\sigma_{1,v})C_v(\sigma_{2,v})]^{1/n}\\
&=3^{r_1+2r_2}\Bigg[\prod_{v\in\Sigma_{\infty}}C_v(\sigma_{1,v})C_v(\sigma_{2,v})\Bigg]^{1/n}.
\end{align*}
Then \eqref{27} follows from Theorem \ref{A}.
\end{proof}

\bibliographystyle{alpha}

\bibliography{Conductor}

\end{document}